\documentclass[10pt]{amsart}

\usepackage[colorlinks,%
citecolor=blue]{hyperref}

\usepackage{amsfonts,amsbsy,amssymb,amsmath,amstext,amsthm}

\usepackage{bbm}

\usepackage{bussproofs}

\usepackage{color}

\usepackage[mathcal]{euscript}

\usepackage{latexsym}

\usepackage{mathbbol}

\usepackage{mathrsfs}

\usepackage{stmaryrd}

\usepackage{wrapfig}

\usepackage[all]{xy}

\let\oldnocite\nocite
\makeatletter
\renewcommand*{\nocite}[1]{\oldnocite{#1}\Hy@backout{#1}}
\makeatother

\allowdisplaybreaks

\theoremstyle{plain}
\newtheorem{theorem}{Theorem}[section]

\newtheorem{lemma}[theorem]{Lemma}
\newtheorem{corollary}[theorem]{Corollary}

\theoremstyle{definition}
\newtheorem{definition}[theorem]{Definition}
\newtheorem{example}[theorem]{Example}
\newtheorem{remark}[theorem]{Remark}

\newcommand{\two}{\mathbbm 2}
\newcommand{\id}{\mathsf{id}}
\newcommand{\op}{^{\mathsf{op}}}
\def\colim#1#2{{\mathsf{colim}_{#1}{#2}}}
\def\sup#1#2{{\mathsf{sup}_{#1}{#2}}}

\newcommand{\V}{\Omega}
\def\vcat{\V\hspace{0pt}\mbox{-}\hspace{.5pt}{\mathsf{cat}}}

\def\kat#1{{\mathscr{#1}}}
\def\A{\kat{A}}
\def\B{\kat{B}}
\def\K{\kat{K}}

\newcommand{\set}{\mathsf{Set}}

\newcommand{\pre}{\mathsf{Ord}}

\newcommand{\vsup}{\V\hspace{0pt}\mbox{-}\hspace{.5pt}\mathsf{Sup}}
\newcommand{\vinf}{\V\hspace{0pt}\mbox{-}\hspace{.5pt}\mathsf{Inf}}
\newcommand{\CCD}{\V\hspace{0pt}\mbox{-}\hspace{.5pt}\mathsf{CD}}

\newcommand{\D}{\mathbbm D}
\newcommand{\U}{\mathbbm U}
\newcommand{\y}{\mathsf{y}}

\def\tensor{\otimes}

\newcommand{\supa}{\bigvee}
\newcommand{\infa}{\bigwedge}

\renewcommand{\phi}{\varphi}

\begin{document}

\title[Equational enriched distributivity]{An equational approach to enriched distributivity}

\author[A. Balan]{Adriana Balan}
\address{University Politehnica of Bucharest, Romania}
\email{adriana.balan@upb.ro}

\author[A. Kurz]{Alexander Kurz} 

\address{Chapman University, Orange California, USA}
\email{akurz@chapman.edu}

\maketitle


\begin{abstract} 
The familiar adjunction between ordered sets and completely distributive lattices can be extended to generalised metric spaces, that is, categories enriched over a quantale (a lattice of ``truth values''), via an appropriate distributive law between the ``down-set'' monad and the ``up-set'' monad on the category of quantale-enriched categories. %
If the underlying lattice of the quantale is completely distributive, and if powers distribute over non-empty joins in the quantale, then this distributive law can be concretely formulated in terms of operations, equations and choice functions, similar to the familiar distributive law of lattices. %

\medskip

  
\medskip
  
\end{abstract}


\section{Introduction}

\subsection{\bf The double dualization monad.} 
The contravariant adjunction 
\[
[-, \two]\dashv [-,\two]:\mathsf{Ord}\op \to \mathsf{Ord}
\]
produces a monad on $\mathsf{Ord}$, the category of antisymmetric ordered sets and monotone maps, whose algebras are the {\em completely distributive lattices}~\cite{MarmolejoRosebrughWood02}. 
The latter are monadic also over $\set$~\cite{PedicchioWood99}, hence form an {\em(infinitary) variety}. Recall that a completely distributive lattice $A$ is a complete lattice satisfying 
\[
\bigwedge_{k\in K} \bigvee A_k = \bigvee_{f\in \mathcal F} \bigwedge f(K)
\]
for every family of subsets $(A_k)_{k\in K}$ of $A$, where $\mathcal F =\{f:K \to A \mid f(k)\in A_{k} \}$ denotes the set of choice functions. There is also a constructive definition of completely distributive lattices available in arbitrary toposes~\cite{MarmolejoRosebrughWood02}: 
if $\D A$ denotes the usual lattice of downsets of an ordered set $A$, and $\y_A: A\to \D A$ is the (Yoneda) embedding of principal downsets, then $A$ being complete means that $\y_A$ has a left adjoint $\sup{A}{}$ which computes suprema of downsets, and being completely distributive means that $\sup{A}{}$ has itself a left adjoint mapping each element to the downset of elements totally below it.\footnote{
An element $a\in A$ is {\em totally below} $b\in A$ if for every subset $S \subseteq A$ such that $b\leq \bigvee S$ there is some $x\in S$ with $a\le x$~\cite{GHKLMS03}.}

As observed in~\cite{RosebrughWood04}, the dual adjunction between $\pre$ and $\pre\op$ exhibited above is in fact a special case of $[-,\V]\dashv [-,\V]:\vcat_{s}\op \to \vcat_{s}$, for $\V=\two$. Here $\vcat_s$ is the category of skeletal categories enriched over a commutative quantale $\V$. 
As in the ordered case, 
the category of algebras for the induced monad can be equivalently described using a distributive law between the enriched downset monad $\D$ and the enriched upset monad 
$\U$~\cite{MarmolejoRosebrughWood02,BabusKurz16,Stubbe17} and  is monadic both over $\vcat_s$ and over $\set$~\cite{PuZhang15}.

\subsection{\bf Quantales.} Historically, quantales were introduced in the 1980s~\cite{Mulvey86} as a logical-theoretic framework for studying certain
spaces arising from quantum mechanics. 
They can be perceived as lattices of ``truth-values'' or ``distances'', equipped with an extra operation expressing conjunction (logical interpretation) or addition of distances (metric interpretation). 
Alternatively, quantales are also complete idempotent semirings, hence appearing in tropical and idempotent analysis.

\subsection{\bf Quantale-enriched categories.} These structures generalise both metric \linebreak spaces and ordered sets~\cite{Lawvere73} within the realm of enriched category theory, making possible a theory of quantitative domains~\cite{Wagner94,Stubbe07a}. 
The category $\vsup$ of {\em (co)complete} skeletal quantale-enriched categories and cocontinuous functors, however, lives in the algebraic world: the objects are just complete sup-lattices endowed with an action of the quantale, or, in other words, complete semimodules over complete idempotent semirings~\cite{JoyalTierney84,PedicchioTholen89,Stubbe07}.
The simplest example, namely cocomplete enriched categories over the two-element quantale $\two$, are nothing but complete sup-lattices. 
It is the action of the quantale which enhances the path to the many-valued realm.

\subsection{\bf Completely distributive quantale-enriched categories.} Adding one \linebreak more layer of structure finally brings us to the object of study of the present paper: $\CCD$, the category of completely distributive skeletal quantale-enriched categories and continuous and cocontinuous enriched functors. There are (at least) three possible approaches/motivations for studying $\CCD$:
\begin{itemize}
\item Completely distributive quantale-enriched categories arise naturally in quantitative domain theory and many-valued logics (see e.g.~\cite{Wagner94,Waszkiewicz09,HofmannWaszkiewicz11,GalatosGilFerez17}).
\item Categorically, they are (co)complete skeletal quantale-enriched categories for which taking suprema distributes over limits, e.g.~\cite{LaiZhang06,Stubbe07a}, and they are also the projective objects of the (infinitary) variety $\vsup$. 
\item Algebraically, completely distributive quantale-enriched categories are precisely the algebras for the quantale-enriched version of the double dualisation monad described above over $\vcat_s$~\cite{BabusKurz16,Stubbe17}. 
Moreover, $\CCD$ is also {\em monadic} over $\set$~\cite{PuZhang15}. 
Consequently, $\CCD$ is equationally presentable~\cite{Linton66}. But Linton's theorem does not help us to find a convenient equational presentation.  

\end{itemize}
The present paper contributes with an equational presentation for completely distributive quantale-enriched categories under some natural requirements on the quantale which, in particular, are satisfied in the case $\V=\two$.


\section{Preliminaires}
\label{sec:preliminaries}

In this section we gather the necessary technicalities from quantale enriched category theory
that will make this paper reasonably self-contained. 
Most of the material is standard; for the general theory of enriched categories we refer to Kelly's monography~\cite{kelly:book}, 
while for quantale/quantaloid-enriched categories the reader might consult e.g.~\cite{Wagner94,Stubbe05,Stubbe06,LaiZhang06,Stubbe07,Stubbe07a,PuZhang15,Hohle15}.


\subsection{Quantales and quantale-enriched categories.}

A {\em commutative quantale} is a complete sup-lattice $(\V,\lor,\land,\bot,\top)$ and a commutative monoid $(\V, \tensor,e)$,%
\footnote{
The unit $e$ does not need to be the top element in $\V$, but if it is the case, the quantale is called {\em integral}. } %
such that $v\tensor-$ preserves arbitrary joins, for all $v\in \V$. In particular it has a right adjoint $[v,-]$:
\[
v \tensor w \leq u \quad \iff \quad w \leq [v,u] \ \ .
\]
If we interpret $v\tensor w$ as a ``conjunction'' and $[v,w]$ as an ``implication'', then $v\tensor[v,u]\leq u $ becomes the usual ``modus ponens'' of logic. Hence the elements of the quantale can be perceived as ``truth values''.

\begin{example}\label{ex:quantales}

\begin{enumerate}

\item\label{ex:2-quantale} 
The simplest example of a (commutative) quantale is the two-element chain $\V = (\two, \land,1)$, with meet as multiplication.

\item\label{ex:[0,infty]-quantale}
The (extended) positive real numbers $\V=([0, \infty]\op, +,0)$, with addition and zero, form a commutative quantale, with $[v,u] = u{-}v \mbox{ \sf{if} } v\leq u \mbox{ \sf{else} }0$.

\item\label{ex:3-quantale} The three-element chain $\mathbb 3=\{0<\frac{1}{2}<1\}$ supports two quantale structures for which the tensor is idempotent~\cite{CasleyCrewMeseguerPratt91}: 
Taking $\tensor$ to be the meet, one obtains the Heyting algebra $\V=(\mathbb{3}, \land, 1)$.

\noindent
\begin{minipage}{.65\textwidth}
The second idempotent multiplication $\otimes$ on $\mathbb 3$ has unit $\frac{1}{2}$ and $[-,-]$ as indicated in the table on the right. In particular, the resulting quantale $(\mathbb 3, \tensor,\frac{1}{2})$ is a \linebreak {\em Sugihara monoid}~\cite{OlsonRaftery07}. 
\end{minipage}
\begin{minipage}{.25\textwidth} 
\vskip-1.2em
{\small{\[
\begin{array}{|c||c|c|c|c|}
\hline 
[-,-] & 0 & \frac{1}{2} & 1 \\ \hline \hline
0     & 1 & 1 & 1 \\
\frac{1}{2}     & 0 & \frac{1}{2} & 1 \\
1     & 0 & 0 & 1 \\
\hline
\end{array}
\]}}
\end{minipage}
%

\noindent
Besides the two idempotent structures exhibited above, there exists only one more quantale structure on $\mathbb 3$ (non-idempotent and integral), given by the {\em {\L}ukasiewicz} truncated addition $v \odot w = \max(0, v+w-1)$. 

\item \label{exle:Jipsen} There are also {non-distributive} commutative quantales, among which we mention the simplest, namely the lattices $M_3$ and the $N_{5}$, both with idempotent tensor as indicated below 
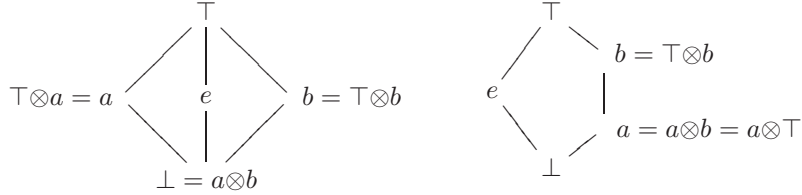
\[
\xymatrix@R=20pt@C=7pt{& \top \ar@{-}[dl] 
\ar@{-}[d] \ar@{-}[dr]& 
\\
{}\save[]+<-0.8cm,0cm>*\txt<8pc>{%
     $\top {\tensor} a = a $}
\restore
& 
e \ar@{-}[d] 
& 
{}\save[]+<0.8cm,0cm>*\txt<8pc>{%
     $b= \top {\tensor} b $}
\restore
\\
& \bot=a{\tensor} b \ar@{-}[ul] \ar@{-}[ur] & }
\xymatrix@R=6pt@C=10pt{
& \top \ar@{-}[ddl] \ar@{-}[dr]& 
\\
&& {}\save[]+<0.8cm,0cm>*\txt<8pc>{%
     $b= \top {\tensor} b $}
\restore
\ar@{}[dd]^(.15){}="a"^(.9){}="b" \ar @{-}"a";"b"
\\
e \ar@{-}[ddr] &&
\\
& & {}\save[]+<1.4cm,0cm>*\txt<8pc>{%
     $a = a {\tensor} b = a {\tensor} \top $}
\restore \ar@{-}[dl] 
\\
& \bot & } 
\]
For more examples we refer to~\cite{GalatosJipsen}. 

\end{enumerate}

\end{example}


An {\em $\V$-category} $\A$ consists of a set $A$, together with a map%
\footnote{
In the sequel we shall refer to this map as the $\V$-hom, $\V$-metric or $\V$-distance.} %
$\A:A\times A \to \V$ 
satisfying 
\begin{equation}\label{eq:vcat}
e\leq \A(a,a)
\qquad \mbox{ and } \qquad 
\A(a,b) {\tensor} \A(b,c) \leq \A(a,c)
\end{equation} for all $a,b,c\in A$. 
An {\em $\V$-functor} $f:\A\to \B$ is a map between the underlying sets such that $\A(a,b)\leq \B(f(a),f(b))$ holds. 
Finally, an $\V$-natural transformation $f\to g$ is given whenever $
e\leq\B(f(a),g(a))
$
holds for all $a\in A$. 
Thus, there is at most one $\V$-natural transformation between $f$ and $g$.
We shall denote by $\vcat$ the category of $\V$-categories and of $\V$-functors (actually, it is a locally ordered category). As a last piece of notation, let $[\A,\B]$ be the $\V$-category having as objects $\V$-functors $f:\A \to \B$, with $\V$-homs $[\A,\B](f,g) = \bigwedge_{a} \B(f(a),g(a))$.

\begin{example}\label{ex:vcats}

\begin{enumerate}

\item\label{ex:V as a V-category} 
The quantale $\V$ becomes an $\V$-category with $\V(v,w) = [v,w]$.

\item\label{ex:discr-vcat} 
Each set $A$ can be perceived as a an $\V$-category $\mathsf{d}A$ when it is equipped with $\mathsf{d}A(a,a) = e$ and $\mathsf{d}A(a,b)=\bot$ for $a\neq b$.  
Such an $\V$-category is called \emph{discrete}. With the obvious action on arrows, this extends to a functor $\mathsf{d}:\set \to \vcat$.

\item 
{\em Ordered sets} $(A,\leq)$ are enriched categories over the two-element quantale $\mathbb 2$: the reflexivity and transitivity axioms 
correspond exactly to~\eqref{eq:vcat}~\cite{Lawvere73}\label{ex:lawvere}. 
%

\item {\em Generalised metric spaces}
If $\V$ is the real half line $([0,\infty]\op, +, 0)$ as in Example~\ref{ex:quantales}.\ref{ex:[0,infty]-quantale}, a small $\V$-category $\A$ is a {\em generalised metric space\/}:
\[
0\geq \A(a,a), \quad \A(a,b) + \A(b,c) \geq \A(a,c) \ \ .
\]
The generalisation of the usual notion is three-fold. 
First, $\A(-,-):A \times A \to [0,\infty]$ is a pseudo-metric in the sense that two distinct points may have distance $0$.
Second, $\A(-,-)$ is a quasi-metric in the sense that distance is not necessarily symmetric. 
Third, distances are allowed to be infinite, which has the important consequence that the category of generalised metric spaces has colimits (whereas metric spaces do not even have coproducts).
An $\V$-functor is then exactly a {\em non-expanding map\/}. 
\end{enumerate}	
\end{example}

In view of the above examples, it is helpful to think of $\V$-categories as $\V$-valued orders~\cite{Hohle15}. 
Actually, each $\V$-category $\A$ does carry an induced order\footnote{
That is, a reflexive and transitive relation.} 
\begin{equation}\label{eq:order}
a\leq b \qquad  \iff \qquad e \leq \A(a,b) \ . 
\end{equation} 
In particular, the underlying order induced by the $\V$-category structure on $\V$ itself, as in Example~\ref{ex:quantales}.\ref{ex:V as a V-category}, is the original order of the quantale. 
An $\V$-category is said to be {\em skeletal} if its underlying order is anti-symmetric: $\A(a,b)\wedge\A(b,a)\geq e$ implies $a=b$. 
Let $\vcat_s$ denote the full subcategory of skeletal $\V$-categories. 

Notice that we use the same symbol $\leq$ for inequality in both $\V$ and $\A$, and rely on the context to tell them apart.
We shall proceed similarly for joins and meets in the underlying order of an $\V$-category, whenever these exist.


\subsection{A calculus of enriched downsets and upsets.}

An $\V$-functor $\phi:\A\op\to \V$, usually called a (contravariant) presheaf in category theory, can be perceived as an $\V$-valued downset:  
the relation $\A(a,b) \leq [\phi(b), \phi(a)]$, equivalent to $\A(a,b) \tensor \phi(b) \leq \phi(a)$, reads in case $\V=\two$ as 
$(a\leq b$ and $b\in \phi$ implies $a\in \phi)$, that is, $\phi$ is a downset in the usual sense. 
Here we implicitly identified a downset with its associated characteristic function. 
To preserve this intuition, we shall denote by $\D\A$ the $\V$-category of contravariant presheaves $[\A\op, \V]$. Notice that $\D\A$ is skeletal, for its underlying order is pointwise, inherited from $\V$. 
The $\V$-functor known as the {\em Yoneda embedding}, $\y_\A:\A \to \D\A$, $\y_\A(a)=\A(-,a)$, generalises the embedding of an ordered set into the lattice of its downsets.  
The correspondence $\A \mapsto \D\A$ extends to a functor $\vcat \to \vcat$: for each $\V$-functor $f:\A \to \B$, put $\D f (\phi) = 
\bigvee_{a \in \A} \B(-,f(a)) \tensor \phi(a)$.
In fact, $\D$ is a {\em Kock-Z\"{o}berlein monad}~\cite{MarmolejoRosebrughWood02}, with unit $\y_\A:\A \to \D\A$ and multiplication $\mu_\A (\Phi) =
\bigvee_{\phi \in \D\A} \phi \tensor \Phi(\phi)$, 
hence it satisfies $\D\y_\A \dashv \mu_\A\dashv \y_{\D\A}$.

\medskip

Dually, covariant presheaves, that is, $\V$-functors $\psi:\A \to \V$, are the $\V$-enriched analogue of upsets; put $\U\A = [\A, \V]\op$ (the $\mathsf{op}$ corresponding to the containment order for usual upsets), and $\y\prime_\A:\A \to [\A,\V]\op$, $\y\prime_\A(a) = \A(a, -)$ the other Yoneda embedding.

\medskip

Given an $\V$-category $\A$, to each ordinary map $f:A \to \V$ (an $\V$-subset\footnote{
For $\V=[0,1]$, this corresponds to the usual notion of a fuzzy subset~\cite{Zadeh65,Goguen67}.} of $A$), %
one can associate an $\V$-downset $f^\downarrow:\A\op \to \V$ by the formula
\begin{equation}\label{eq:down-closure}
f^\downarrow = \bigvee_{a\in A} f(a) \tensor \A(-,a) \ \ .
\end{equation}
Readers familiar with category theory will recognise in the above formula the left Kan extension of $f$ along the inclusion of the discrete $\V$-category $\mathsf{d}A$ into $\A$; 
in the particular case $\V=\two$, $f^\downarrow$ is the {\em down-closure} of a subset $f$.

\smallskip

\begin{minipage}{.75\textwidth}
\begin{example}
Consider the {\L}ukasiewicz chain $\V=\mathbb 3$, \linebreak together with the singleton subset $\{\frac{1}{2}\}$ represented via its characteristic map as the discrete $\V$-subset $f:\V \to  \V$, $f(\frac{1}{2})=1, f(0)=f(1)=0$. \end{example}
\end{minipage}
\begin{minipage}{.25\textwidth}
\vskip-1.5em
\renewcommand{\arraystretch}{1.3}
{\small{
\begin{equation*}
\begin{array}{|c||c|c|c|}
\hline
a & 0 & \frac{1}{2} & 1\\
\hline 
f^{\downarrow}(a) & 1 & 1 & \frac{1}{2} \\
\hline
\end{array}
\end{equation*}
}}
\end{minipage}
Then its ordinary downclosure is $\{0,\frac{1}{2}\}$, while the $\V$-enriched
 $f^{\downarrow}$ is computed in the next table: in the enriched context, the $\V$-valued downclosure $f^\downarrow$ of the subset $\{a\}$ carries more information (it has a richer structure) than the ordinary one.


\subsection{Limits and colimits in $\V$-categories. (Co)completeness of $\V$-categories.}

In this subsection we shall briefly recall the usual notions of weighted colimits and cocompleteness to be used in the sequel. 

\medskip

A {\em colimit\/} of an $\V$-functor $f:\K\to\A$, 
weighted by an $\V$-functor $\phi:\K\op \to\V$, consists of an object $\colim{\phi}{f}$ of $\A$, such that 
\begin{equation}\label{def:colimit}
\A(\colim{\phi}{f},a) = [\K\op,\V](\phi,\A(f-,a)) = \bigwedge_{k\in \K} [\phi(k),\A(f(k),a)]
\end{equation}
holds, $\V$-natural in $a\in \A$.

An $\V$-category is called {\em cocomplete} if it has all $\V$-enriched colimits; an $\V$-functor is {\em cocontinuous} if preserves all $\V$-enriched colimits. Denote by $\vsup$ the category of cocomplete skeletal $\V$-categories and cocontinuous $\V$-functors. 

It will be convenient for us to use the alternative characterisation of a cocomplete $\V$-category as one having all colimits of the identity functor~\cite{Stubbe05}. That is, for each $\phi:\A\op \to \V$ there is an object $\colim{\phi}{\id_\A}$ in $\A$ such that 
\[
\A(\colim{\phi}{\id_\A},a) = \D\A(\phi,\y_\A(a)) = \bigwedge_{b\in \A} [\phi(b),\A(b,a)] \ \ .
\]
Interpreting the above formula in the ordered case, we see that  $\colim{\phi}{\id_\A} \leq a$ holds if and only if $b \in \phi $ implies $b \leq a$ for all $b$. Hence, intuitively, $\colim{\phi}{\id_\A}$ computes the $\V$-suprema of the $\V$-contravariant presheaf $\phi:\A\op \to \V$ (think again of $\phi$ as an "$\V$-valued downset") and will be denoted $\sup{\A}{\phi}$. Then cocompleteness of $\A$ can be rephrased as the property of the Yoneda embedding of having a left adjoint (namely, $\sup{\A}{}:\D\A\to \A$).

\begin{example}
We list below two types of weighted colimits:

\begin{enumerate}

\item If $\K$ is the unit $\V$-category%
\footnote{
Having only one object, with corresponding $\V$-hom given by the unit of the quantale $e$.} %
$\mathbb 1$, we may identify $f$ with an object $a$ of $\A$ and $\phi$ with an element $v$ of the quantale $\V$. 
The resulting colimit, usually called the {\em tensor} (or copower) of $v$ with $a$, will be denoted $v\ast a$ instead of $\colim{\phi}{f}$. 
Explicitly, the tensor $v\ast a$ is uniquely determined by the relation $\A(v \ast a, b) = [v, \A(a,b)] $ for all $a,b\in \A$.

\item A {\em conical colimit} is a colimit weighted by a presheaf $\phi:\K^{\mathsf{op}} \to \V$ with constant value $e$, the unit of the quantale; its defining property is therefore $
\A(\colim{\phi}{f},a) = \bigwedge_{k\in \K} \A(f(k),a)$ for all $a\in \A$. 
In particular, with respect to the induced order on $\A$, $\colim{\phi}{f}\leq a $ holds if and only if $f(k) \leq a$ holds for every $k$; that is, $\colim{\phi}{f}$ is the {\em join} $\bigvee_{k\in \K} f(k)$ of the family $(f(k))_{k\in \K}$ in the underlying ordered set of the $\V$-category $\A$.
\end{enumerate}

\end{example}

The importance of the previous two types of colimits resides in the characterisation of a cocomplete $\V$-category as an $\V$-category having all tensors and all conical colimits~\cite{kelly:book, Stubbe06}. 
For later use, we just mention how the $\V$-suprema of a contravariant presheaf $\phi:\A\op \to \V$ is obtained via joins and tensors:
\begin{equation}\label{eq:sup-as-join-of-tensor}
\sup{\A}{\phi} = \bigvee_{a\in \A} \phi(a) \ast a \ \ .
\end{equation}

In particular, a skeletal cocomplete $\V$-category is a complete lattice with respect to the induced order.

\begin{remark} \label{rem:-discrete-domain-of-weights}
A careful analysis of the defining property of a weighted colimit~\eqref{def:colimit} shows that without loss of generality, the domain $\V$-category $\K$ of the weight can be chosen discrete.
This is because the quantale in which we enrich is actually a poset. 
The fact that weighted colimits can be discrete means that we can treat them as {\em operations}, and this will be useful in the sequel. 
\end{remark}

Dually, one can define weighted limits (in particular, cotensors $v\rhd a$, meets $\bigwedge_k a_k$ and infima $\inf_\A \psi$ of $\V$-valued upsets $\psi:\A \to \V$) and talk about completeness of $\V$-categories and continuity of $\V$-functors. For more details, we refer to e.g.~\cite{Stubbe05,Stubbe06}, and only recall here the formula allowing to compute the limit of an $\V$-functor $G:\K \to \A$ weighted by $w:\K\to \V$, by meets and cotensors:
\begin{equation}\label{eq:lim-as-meet-of-cotensors}
\mathsf{lim}_{w}G = \bigwedge_{k\in \K} w(k)\rhd G(k) \ .
\end{equation}

\smallskip

To deepen the analogy with order and lattice theory, we add that an $\V$-category $\A$ is complete if and only if it is cocomplete~\cite{Stubbe05}.%
\footnote{
Although the two notions are equivalent, we shall continue to refer to ``cocomplete'' $\V$-categories, as we are mainly interested in suprema of $\V$-downsets and in {\em cocontinuous} $\V$-functors.} %
Alternatively, $\A$ is cocomplete if it is complete as a lattice and has all tensors and cotensors~\cite{Stubbe06}. 
This last result has the advantage of characterising $\vsup$, the category of cocomplete skeletal $\V$-categories and cocontinuous $\V$-functors by operations and equations, as we shall recall next. Formally, the forgetful functor $\vsup \to \set$ is monadic. This goes back to~\cite{JoyalTierney84} and~\cite{PedicchioTholen89}, but see also~\cite{
JanelidzeKelly01,Stubbe07}.


\subsection{\bf Equational presentation of cocomplete $\V$-categories.} A cocomplete \linebreak skeletal $\V$-category can alternatively be described as a complete join-lattice $(A,\supa,\bot)$ endowed with an action of $\V$, $(v,a) \mapsto v \ast a$~\cite{Stubbe06}, satisfying the following equations: 
\begin{equation}\label{eq:vsup}
\begin{aligned}
&
\ e \ast a = a 
&&
v\ast (w \ast a) = (v {\tensor} w) \ast a
\\
&
\left(%
\bigvee_{i\in I} v_i%
\right)%
\ast a = \supa_{i\in I} \left(v_i \ast a \right) \qquad
&&
v \ast \left(%
\supa_{i\in I} a_i%
\right)%
= \supa_{i\in I} \left(v \ast a_i \right) \ .
\end{aligned}
\end{equation}
\medskip

Also, cocontinuous $\V$-functors are precisely the join-lattice homomorphisms preserving the action of $\V$.

\begin{remark}
Readers familiar with ring and module theory will certainly identify cocomplete skeletal $\V$-categories as complete semimodules over $\V$, $\V$ being perceived as a complete idempotent semiring, and cocontinuous $\V$-functors as complete homomorphisms of $\V$-semimodules~\cite{CohenGaubertQuadrat04}. 
\end{remark}

As mentioned above, a (skeletal) $\V$-category $\A$ is complete if and only if it is cocomplete.
In particular, cocompleteness implies existence of arbitrary meets and cotensors. For example, the cotensor $v \rhd - $ is a unary operation on $\A$, right adjoint to $v \ast - $:
\[
v \ast a \leq b \quad \iff \quad a \leq v \rhd b \ .
\]
\vskip.5em

Also, meets and cotensors automatically satisfy equations dual to those in~\eqref{eq:vsup} (formally deduced by adjunction), namely
\begin{equation}\label{eq:vinf}
\begin{aligned}
&
\ e \rhd a = a 
&&
v\rhd (w \rhd a) = (v {\tensor} w) \rhd a
\\
&
\left(%
\bigvee_{i\in I} v_i%
\right)%
\ast a = \infa_{i\in I} \left(v_i \rhd a \right) \qquad
&&
v \rhd \left(%
\infa_{i\in I} a_i%
\right)%
= \infa_{i\in I} \left(v \rhd a_i \right) \ .
\end{aligned}
\end{equation}

The $\V$-category structure can be recovered from the equational presentation above~\cite{PedicchioTholen89} as
\begin{equation}\label{eq:v-hom}
\A(a,b) = \bigvee\{v \mid v \ast a \le b\} = \bigvee\{v \mid a \leq v \rhd b\} \ . 
\end{equation}

\smallskip

For further use, let us also record the relation
\begin{equation}\label{eq:tensor-with-V-hom}
\bigvee_{a\in \A} \A(a,b) \ast a = b
\end{equation}
which generalises the obvious statement that in an ordered set, the join of all elements below some $b$ is $b$ itself.

\begin{remark}\label{rem:eq-pres-vsup}
It will be beneficial in the next section to use an equivalent, but more symmetric presentation of (the objects of) $\vsup$:
the objects are complete lattices (hence having both arbitrary joins $\bigvee$ and meets $\bigwedge$), 
together with a family of adjoint unary operators indexed by $\V$ (tensors $v\ast-$ and cotensors $v\rhd-$), 
satisfying both~\eqref{eq:vsup} and~\eqref{eq:vinf}, in addition to the usual complete lattice equations. 
Notice that arrows in $\vsup$ do not preserve all operations described above, but only joins and tensors\footnote{ 
Consequently, meets and cotensors are only laxly preserved.}.
\end{remark}

\begin{example}

\begin{enumerate}

\item Perhaps the simplest example of a cocomplete $\V$-category is $\V$ itself, with the $\V$-category structure described in Example~\ref{ex:quantales}.\ref{ex:V as a V-category}. Trivially, the unit category $\mathbb 1$ is also cocomplete. 

\item For any $\V$-category $\A$, the associated $\V$-category of $\V$-downsets $\D\A$ is cocomplete and skeletal. In fact, $\D\A$ is the {\em free cocompletion} of $\A$~\cite{kelly:book,Stubbe05}.%
\footnote{
In particular, taking the domain to be discrete, we obtain the cocomplete $[\mathsf{d}A,\V]=\V^A$. The case $A=\emptyset$ produces the cocomplete terminal $\V$-category $\mathbb 1_\top$. } %
For later use, recall how suprema and limits are computed in $\D\A$:
\begin{equation}\label{eq:sup-in-free-cocompletion}
\mathsf{sup}_{\D\A} \Phi = \mu_\A (\Phi) = \bigvee_{\phi\in\D\A} \Phi(\phi) \tensor \phi
\end{equation}
respectively 
\begin{equation}\label{eq:lim-in-free-cocompletion}
\mathsf{lim}_w G = \bigwedge_{k\in \K} [w(k),G(k)(-)]
\end{equation}
where $\Phi:(\D\A)\op \to \V$, $w:\K \to \V$ and $G:\K \to \D\A$ are $\V$-functors.

\item For each $\V$-category $\A$, there is an adjunction commuting with both Yoneda embeddings
\[
\xymatrix@C=40pt@R=15pt{%
&
\A
\ar[dl]_{\y_\A}
\ar[dr]^{{\y\prime}_\A}
&
\\
[\A\op, \V]
\ar@<+1ex>[rr]
\ar@{}[rr]|-\perp 
& 
&
[\A,\V]\op 
\ar@<+1ex>[ll]
}
\]
where the horizontal left adjoint maps an $\V$-downset $\phi$ to 
\[
{\uparrow}\phi = \bigwedge_{a\in \A}[\phi(a),\A(a,-)]
\] 
(the analogue of the (up)set of upper bounds), respectively an $\V$-enriched upset $\psi:\A \to \V$ is mapped by the right adjoint to the $\V$-downset of its lower bounds 
\[
{\downarrow}\psi=
\bigwedge_{a\in \A} [\psi(a), \A(-,a)] \ \ .
\]
%
\noindent The fixed points of these adjunction determine a (complete and) cocomplete skeletal $\V$-category into which $\A$ embeds continuously and cocontinuously, known as the {\em Isbell completion} (the $\V$-categorical analogue of the Dedekind-MacNeille completion of an ordered set)~\cite{Wagner94,Rutten98,Stubbe05}. 
For example, in case $\V=([0,\infty],\geq_{\mathbb R},+,0)$, the Isbell completion is known in the theory of metric spaces as the {\em tight span}~\cite{Willerton13}.

\end{enumerate}

\end{example}


\section{Completely distributive quantale enriched categories}

There are several ways of introducing completely distributive $\V$-categories. Here, we chose to take the perhaps simplest approach, following~\cite{RosebrughWood94}:

\begin{definition}
A {\em completely distributive} $\V$-category%
\footnote{
Also know as a {\em totally continuous} $\V$-category.%
} %
 is a cocomplete $\V$-category $\A$, such that $\sup{\A}{}:\D\A\to \A$ has an $\V$-enriched left adjoint.
\end{definition}

Completely distributive $\V$-categories have been studied in the past by several authors, see e.g.~\cite{LaiZhang06,Stubbe07,BabusKurz16,Stubbe17}. 
As the name suggests, for $\V=\two$ we recover the well-known completely distributive lattices (the left adjoint to $\sup{\A}{}$ mapping an element to the downset of those totally below it). 
We shall denote by $\CCD$ the category of completely distributive skeletal $\V$-categories and continuous and cocontinuous $\V$-functors. 

\begin{remark} 
We shall see in Section~\ref{sec:eq-vccd} how the choice for arrows in $\CCD$ is imposed by the axiomatisation~\eqref{eq:eq-ccd1}: all operations encountered must be preserved. Actually, there is a deeper categorical reason behind that: $\CCD$ is the category of algebras for the lifting of the $\V$-downset monad $\D$ to $\vinf$, the category of algebras for the $\V$-upset monad $\U$~\cite{BabusKurz16,Stubbe17}. 
Consequently, the arrows in $\CCD$ are simultaneously $\U$-homomorphisms and $\D$-homomorphisms, that is, continuous and cocontinuous $\V$-functors. 
\end{remark}

\begin{example}
For any $\V$-category $\A$, $\D\A$ is completely distributive, the left adjoint of $\sup{\D\A}{}$ being $\D\y_\A $~\cite{LaiZhang06,Stubbe07a}.
In particular, for any set $A$, the $\V$-valued powerset $\V^A = [\mathsf{d}A,\V]$ is completely distributive. 
Taking $A$ to be a singleton shows that the quantale $\V$ is itself {\em completely distributive as an $\V$-category}, 
while $A=\emptyset$ produces the completely distributive terminal $\V$-category $\mathbb 1_\top$. 
Also, the unit $\V$-category $\mathbb 1$ is completely distributive.
\end{example}

\begin{remark}
\begin{enumerate}

\item A cocomplete skeletal $\V$-category is completely distributive if and only if it is a projective object in $\vsup$~\cite{Stubbe07}.

\item Unlike the case for lattices, complete distributivity is no longer a self-dual notion in general;\footnote{
Neither is this the case when working in an arbitrary topos~\cite{FawcettWood90}.} %
in fact, in~\cite{LaiZhang06} it is proven that for an integral quantale $\V$, every completely distributive $\V$-category is also  completely co-distributive if and only if $\V$ is a Girard quantale~\cite{Yetter90}. 

\item Complete distributivity of a cocomplete $\V$-category does not necessarily entail the complete distributivity of the underlying lattice. 
For example, $\V$ itself is always completely distributive as an $\V$-category~\cite{Stubbe07a}, but not necessarily distributive as a lattice, as we have seen from Example~\ref{ex:quantales}.
However, there exists a positive result in this sense, due to~\cite{LaiZhang06}: every completely distributive $\V$-category $\A$ is completely distributive as a lattice if and only if $\V$ itself is a completely distributive lattice.

\noindent
Actually, if the reader is interested in cocomplete $\V$-categories which are not completely distributive, there is a simple way of producing such examples: take any complete lattice $A$ which {\em is not completely distributive}, e.g. the diamond lattice $M_{3}$, and a quantale $\V$ which {\em is completely  distributive} as a lattice.
Then the tensor product of $A$ and $\V$ in the category of complete sup-lattices is a cocomplete $\V$-category -- it is the free cocomplete $\V$-category over the complete sup-lattice $A$~\cite{JoyalTierney84,PedicchioTholen89}, but not completely distributive as a lattice~\cite{Shmuely79}. 
Therefore this tensor product is neither $\V$-completely distributive.   
\end{enumerate}

\end{remark}


\subsection{The equational theory of completely distributive $\V$-categories}\label{sec:eq-vccd}

In~\cite{PuZhang15} it is shown that the category $\CCD$ is {\em monadic} over $\set$, hence a(n infinitary) variety~\cite{Linton66}.  
In particular, the free completely distributive $\V$-category over a set $A$ is 
$\D\U (\mathsf{d}A)$~\cite{BabusKurz16,Stubbe17}, similarly to the ordered case~\cite{Markowsky79}.

We shall see in the sequel how $\CCD$ can be described by a nice set of operations and equations, provided certain assumptions on the quantale are fulfilled.
We begin with a simple result, which will turn out to be useful in computations:

\begin{lemma}\label{lemma:rhdstar}
Let $\A$ be a complete $\V$-category. Then 
\[
[v, w]\ast a \ \le\  v\rhd (w\ast a)
\]
holds for any $v,w\in \V$ and $a\in \A$. More generally, we have 
\[
\left( \bigwedge_i [v_i, w_i] \right) \ast a \ \le \ \bigwedge_i \left(\vphantom{\bigwedge_i} v_i \rhd (w_i \ast a) \right)
\]
for all $v,w\in \V$ and all $a\in\A$.
\end{lemma}

\begin{proof}
The following standard sequence of implications below proves the first assertion: 
\def\proofSkipAmount{\vskip -2pt} 
\begin{prooftree}
\AxiomC{$v \otimes [v,w]\leq w$}
\UnaryInfC{$v \ast ([v,w]\ast a) = (v \otimes [v,w])\ast a \leq w\ast a$}
\UnaryInfC{$[v,w]\ast a\leq v \rhd (w\ast a)$}
\end{prooftree}
The second statement is an immediate consequence. 
\end{proof}

\subsection{\bf Assumption.}\phantomsection\label{assumption} 
In what follows, we shall consider $\V$ to be a commutative  quantale, such that 
\begin{enumerate}
\item \label{ass:V-ccd}
$\V$ is {completely distributive} as a lattice.
\item \label{ass:powers-pres-joins} All cotensors $[v,-]:\V \to \V$, for $v\in\V$, preserve {non-empty joins}, that is, 
\[
[v,\bigvee_{i\in I} w_i]=\bigvee_{i\in I} [v,w_i]
\] 
holds in $\V$ for non-empty $I$. 
\end{enumerate}

\begin{remark}
\begin{enumerate}

\item 
Observe that in any quantale, 
$[\bot,v]=\top$ holds for all $v$, hence in particular $[\bot,\bot]=\top$.\footnote{ %
Hence $[\bot, -]$ preserves the empty join only if the quantale is trivial.} %
Also,  
$[v,\bot]=\bot$ 
holds for all 
$v\neq \bot$ in $\V$ 
if and only if the quantale has no zero divisors.\footnote{
That is, $v\otimes w=\bot$ implies $v=\bot$ or $w=\bot$.}
Therefore extending Assumption~\ref{ass:powers-pres-joins} as to include the empty join, that is, the bottom element of the quantale, would not be reasonable.  
%


\item 
The above assumptions are satisfied for all quantales in Example~\ref{ex:quantales}.\ref{ex:2-quantale}-\ref{ex:3-quantale} %
and also by all finite semilinear commutative residuated lattices (in particular by the finite commutative MTL-algebras)~\cite{EstevaGodo01,JipsenTsinakis02}.

\end{enumerate}
\end{remark}

\medskip

Using the symmetric presentation of cocomplete $\V$-categories, as emphasised in Remark~\ref{rem:eq-pres-vsup}, we can now state the main result of this paper:

\begin{theorem}\label{thm:main}

Let $\A=(A, \supa,\infa,(v\ast-)_{v\in \V},(v\rhd-)_{v\in \V})$ be a cocomplete skeletal $\V$-category. 
If $\A$ is completely distributive, then it satisfies 
\begin{equation}\label{eq:V-eq-ccd}
{\bigwedge_{k\in K} \psi(k) \rhd \left( \bigvee_{a\in A} G(k)(a) \ast a \right)
=
\bigvee_{f\in \mathcal F} \bigwedge_{k\in K} \psi(k)\rhd \left( \vphantom{\bigvee_{i}}G(k)(fk)\ast fk \right)}
\end{equation}
for every pair of functions $\psi:K\to \V$, $G:K \to \V^{A}$, where $\mathcal F$ denotes the set of functions $K\to A$.%
\end{theorem}

\begin{proof}
Recall that a cocomplete $\V$-category $\A$ is completely distributive if and only if $\sup{\A}{}:\D\A \to \A$ has an $\V$-enriched left adjoint, equivalently, if $\sup{\A}{}:\D\A \to \A$ preserves weighted limits. 
That is, 
\[
\sup{\A}{\left(\mathsf{lim}_\psi G\right)} = \mathsf{lim}_\psi ( \sup{\A}{}\circ G)
\]
holds for every $\V$-functors $\psi:\K \to \V$ and $G:\K\to \D\A$.
Expressing $\sup{\A}{}$ by tensors and joins, and the weighted limits above by cotensors and meets as in~\eqref{eq:sup-as-join-of-tensor},~\eqref{eq:lim-as-meet-of-cotensors},~\eqref{eq:sup-in-free-cocompletion},~\eqref{eq:lim-in-free-cocompletion}, we obtain 
\begin{equation}\label{eq:vccd-cond}
\bigvee_{a} %
\left( \bigwedge_k [\psi(k),G(k)(a)] \right) \ast a %
= %
\bigwedge_k %
\psi(k) \rhd \left(\vphantom{\bigvee_i} \bigvee_a G(k)(a) \ast a\right)  \ . 
\end{equation}
As explained in Remark~\ref{rem:-discrete-domain-of-weights}, we can {\em replace} $\K$ by a {\em discrete $\V$-category (a set)} without loss of generality.
Hence instead of an $\V$-functor $\psi:\K \to \A$, we shall consider a mere function $\psi:K \to A$.

The natural next step will then be to substitute {\em the $\V$-functor} $G:\K \to \D\A=[\A\op, \V]$ by {\em a function} $G:K\to \V^A = [A , \V]$.
But in this case there is a price to pay: the passage from the family of $\V$-downsets $(G(k):\A\op \to \V)_k$ to a family of $\V$-subsets $(G(k):A \to \V)_k$ forces the appearance of {the $\V$-down-closure} of each $\V$-subset $G(k)\in \V^A$ in~\eqref{eq:vccd-cond}, namely
\begin{equation}\label{eq:eq-ccd1}
\bigvee_a %
\left( \bigwedge_k [\psi(k),{G(k)^{{\downarrow}}(a)}] \right) \ast a %
= %
\bigwedge_k %
\psi(k) \rhd \left( \bigvee_a G(k)^{\downarrow}(a) \ast a \right) 
\end{equation}
But since ``the supremum of an $\V$-subset is the supremum of its $\V$-down-closure'', that is, 
\begin{align*}
&
\bigvee_a G^{\downarrow}(k)(a) \ast a 
&
= 
&&&
\bigvee_a \left(\bigvee_b G(k)(b) \otimes \A(a,b) \right)\ast a
\\
&
&
=
&&&
\bigvee_{b} G(k)(b) \ast \left(\bigvee_a \A(a,b) \ast a \right) 
\\
&
&
=
&&&
\bigvee_{b} {G(k)(b) \ast b \vphantom{\bigvee_a}} \ \mbox{\ \ (by~\eqref{eq:tensor-with-V-hom})}
\end{align*}
we see that the right hand side of Equation~\eqref{eq:eq-ccd1} above remains unchanged by substituting $G(k)^{\downarrow}$ by $G(k)$ and that the complete distributivity equation~\eqref{eq:vccd-cond} becomes 
\begin{equation}\label{eq:eq-ccd2}
\bigvee_a %
\left( \bigwedge_k [\psi(k),{\bigvee_{b} G(k)(b){\tensor} \A(a,b)}] \right) \ast a %
= %
\bigwedge_k %
\psi(k) \rhd \left( \bigvee_a G(k)(a) \ast a \right) \, .
\end{equation}
Notice that the left hand side involves now not only the tensor and the internal hom of the quantale, but also the $\V$-category structure of $\A$, which, as expressed in Equation~\eqref{eq:v-hom}, leads beyond equational logic.
At this stage, the assumptions on $\V$ come into place for the following sequence of computations:
\begin{align*}
&
&&
\bigvee_a \left( \, \bigwedge_{k} \ [\psi(k),\bigvee_b G(k)(b) \otimes \A(a,b)] \right) \ast a
\\
&
=
&&
\bigvee_a \left( \, \bigwedge_{k} \, \bigvee_{b} \ [\psi(k),G(k)(b) \otimes \A(a,b)] \right) \ast a 
&& 
\mbox{(\hyperref[ass:powers-pres-joins]{by Assumption.(2)})}%
\\
&
=
&&
\bigvee_a \left( \bigvee_{f} \bigwedge_{k} \ [\psi(k),G(k)(fk) \otimes \A(a,fk)] \right) \ast a
&& 
\mbox{(\hyperref[ass:V-ccd]{by Assumption.(1)})}
\\
&
\leq
&&
\bigvee_{a,f}  \bigwedge_{k}  \psi(k) \rhd  
\left( 
\vphantom{\bigvee_i} 
\left(
\vphantom{\bigvee_i} 
G(k)(fk) \otimes \A(a,fk)
\right)
\ast a \right)
&&
\mbox{(by Lemma~\ref{lemma:rhdstar})}
\\
&
=
&&
\bigvee_{a,f}  \bigwedge_{k} 
\psi(k)\rhd 
\left(
\vphantom{\bigwedge_{k}} 
G(k)(fk) \ast 
\left(
\vphantom{\bigwedge_{k}} 
\A(a,fk) \ast a
\right) 
\right)  
\\
&
\leq
&&
\bigvee_{f}  \bigwedge_{k} \left(\vphantom{\bigwedge_{k}} \psi(k)\rhd \left(\vphantom{\bigwedge_{k}} G(k)(fk) \ast (\bigvee_{a} \A(a,fk) \ast a) \right)  \right)
\\
&
=
&&
\bigvee_{f}  \bigwedge_{k} \psi(k)\rhd \left(\vphantom{\bigwedge_{k}}  G(k,fk) \ast fk \right)
&&
\mbox{(by \eqref{eq:tensor-with-V-hom})}
\end{align*}
Therefore $\displaystyle{\bigwedge_k %
\psi(k) \rhd \left( \bigvee_a G(k)(a) \ast a \right) \leq \bigvee_{f}  \bigwedge_{k} \psi(k)\rhd \left( \vphantom{\bigwedge_{k}}  G(k)(fk) \ast fk \right)}$ holds. 
The opposite inequality is trivial, as for each $f\in \mathcal F$ we always have
\[
\bigwedge_k \psi(k)\rhd \left(\vphantom{\bigvee_a} G(k)(fk)\ast fk\right)
\leq
\bigwedge_k \psi(k)\rhd \left(\bigvee_{a} G(k)(a)\ast a\right)
\]
hence 
\begin{equation}\label{eq:trivial-ineq}
\bigvee_{f}\bigwedge_k \psi(k)\rhd \left( \vphantom{\bigvee_i} G(k)(fk)\ast fk \right)
\leq
\bigwedge_k \psi(k) \rhd \left( \bigvee_a G(k)(a) \ast a\right)
\end{equation}
and the proof is finished.
\end{proof}


\begin{remark}\label{rem:vccd}
\begin{enumerate}

\item\label{rem:vccd=>ccd} We already know from~\cite{LaiZhang06} that assuming the quantale to be completely distributive entails that each completely distributive $\V$-category is also completely distributive as a lattice. 
To see Theorem~\ref{thm:main} at work, we shall recover the cited result of Lai and Zhang by choosing trivial weights $\psi(k)=e$ and discrete $\V$-subsets $G(k)$ corresponding to a family of ordinary subsets $(A_k)_{k\in K}$ of $A$; that is, 
\[
G(k)(a) = 
\begin{cases} 
e \mbox{ for } a\in A_k 
\\
\bot \mbox{ otherwise } 
\end{cases} 
\forall \, k
\]
Then Equation~\eqref{eq:V-eq-ccd} becomes 
\[
{\bigwedge_{k}  \bigvee A_k
=
\bigvee_{\{f:K\to A \, \mid \, fk\in A_k\}}  \bigwedge_{k} fk} \ .
\]
The reader will recognise the usual law of complete distributivity via choice functions.

\item\label{rem:V-gen} Again by~\cite{LaiZhang06}, each completely distributive $\V$-category is a homomorphic image %
of a subobject of a product of copies of $\V$.\footnote{%
This is the $\V$-enriched generalization of Raney's well-known result that a complete lattice is completely distributive if and only if it is the homomorphic image of a ring of sets~\cite{Raney52}.}
This can be immediately seen as follows: first, any completely distributive $\V$-category $\A$ is a retract of $\D\A$ via $\sup{\A}{}:\D\A \to \A$, which is both continuous and cocontinuous, hence $\A$ is a quotient of $\D\A$ in $\CCD$.
Second, the inclusion functor $\D\A \to [\mathsf{d}A,\V]=\V^A$ is both continuous and cocontinuous (with adjoints provided by Kan extensions), hence $\D\A$ is a subobject of $\V^A$ in $\CCD$. 
That is, $\CCD = \mathsf{HSP}(\V)$. 
In particular, an equation holds in a completely distributive $\V$-category $\A$ if and only if it holds in $\V$. This is perhaps best illustrated by the following Corollary (compare also with~\cite[Corollary~4.14]{LaiZhang20}):

\end{enumerate}

\end{remark}

\begin{corollary}\label{cor:easy-vccd}
Let $\A=(A, \supa,\infa,(v\ast-)_{v\in \V},(v\rhd-)_{v\in \V})$ be a cocomplete skeletal $\V$-category. Then $\A$ satisfies~\eqref{eq:V-eq-ccd} if and only if $\A$ is completely distributive as a lattice and cotensors in $\A$ preserve non-empty joins:
\begin{equation}\label{eq:contensor/joins}
v\rhd  \supa_{i} a_i = \supa_{i} (v \rhd a_i)
\end{equation}
\end{corollary}

\begin{proof}
The complete distributivity of $\A$ as a lattice has been discussed above. Take now $K=\{0\}$, $\psi(0)=v$, $G(k)(a_i)=e$ and $G(k)(-)=\bot$ otherwise to obtain~\eqref{eq:contensor/joins}.  
Conversely, using that cotensors preserve non-empty joins, that each complete $\V$-category is non-empty, and finally the complete distributivity of the underlying lattice $(\A, \bigvee, \bigwedge)$, we have that 
\begin{align*}
&
\bigwedge_{k} \psi(k) \rhd \left( \bigvee_{a} G(k)(a) \ast a \right) 
&&
= 
&&
\bigwedge_{k} \bigvee_{a} \left(\vphantom{\bigvee_i} \psi(k) \rhd \left(  G(k)(a) \ast a \right) \right)
\\
&
&&
=
&& 
\bigvee_{f} \bigwedge_{k} \left(\vphantom{\bigvee_i} \psi(k)\rhd(G(k)(fk)\ast fk) \right)
\end{align*}
for each functions $\psi:K\to \V$, $G:K\to \V^A$, hence we recover the complete distributivity law with obtained in Theorem~\ref{thm:main}.
\end{proof}

%

\begin{remark}
Earlier we mentioned that the equational theory of $\CCD$ is generated by $\V$, regardless of~\hyperref[assumption]{Assumption}. Hence there is no hope for obtaining various variants of distributivity laws (like the one discussed below) unless these already hold in $\V$. 

For example, observe that the parentheses in expressions involving tensors and cotensors, e.g. $v\rhd(w \ast a)$, do not change in the above law~\eqref{eq:V-eq-ccd} of complete $\V$-distributivity. 
One may then ask whether the distributivity-like property
\[
v\rhd(w \ast a)= [v,w]\ast a
\] 
could be of interest for $\CCD$. 
But the corresponding relation for $\V$ itself, namely $
[v,w {\tensor} u] = [v,w]{\tensor} u $, implies that the quantale is {\em trivial}. 
To see this, observe that $e\leq [v,v] = [v,e\otimes v] = [v,e] \otimes v \leq e$. Therefore $\V$ is an ordered abelian group~\cite{Lambek99}, hence it cannot be bounded as a lattice unless it is trivial~\cite{Birkhoff79}.

Consequently, completely distributive $\V$-categories are not so distributive after all.  \hspace*{\fill}\qed
\end{remark}

Before ending the paper, we point out one more issue: comparing the constructive versus the non-constructive equations of $\V$-completely distributivity, namely~\eqref{eq:eq-ccd2} and~\eqref{eq:V-eq-ccd}, as recalled below, we see that the left hand sides coincide.
\setlength{\abovedisplayskip}{2pt}
\setlength{\belowdisplayskip}{2pt}
\begin{equation}\label{eq:Vcd-versus-Vccd}
\begin{aligned}
& \bigwedge_{k\in K} %
\psi(k) \rhd \left( \bigvee_{a\in A} G(k)(a) \ast a \right)  
&
=
&&
{\bigvee_{a\in A} %
\left( \bigwedge_{k\in K} [\psi(k),G(k)^{\downarrow}(a)] \right) \ast a }%
\\
&
\bigwedge_{k\in K} \psi(k) \rhd \left( \bigvee_{a\in A} G(k)(a) \ast a \right)
&
=
&&
{\bigvee_{f\in \mathcal F} \bigwedge_{k\in K} \psi(k)\rhd(G(k)(fk)\ast fk)} \ . 
\end{aligned}
\end{equation} 

A careful inspection of the proof of Theorem~\ref{thm:main} reveals that the inequality 
\setlength{\abovedisplayskip}{6pt}
\setlength{\belowdisplayskip}{4pt}
\[
\bigvee_{a\in A} %
\left( \bigwedge_{k\in K} [\psi(k),G(k)^{\downarrow}(a)] \right) \ast a 
\leq
\bigvee_{f\in \mathcal F} \bigwedge_{k\in K} \psi(k)\rhd(G(k)(fk)\ast fk)
\]
holds in $\CCD$ if~\hyperref[assumption]{Assumption} holds. We shall now see an example where the inequality is not true.

\begin{example} 
Take one of the two non-distributive quantales $M_3$ and $N_5$ from Example~\ref{ex:quantales}\ref{exle:Jipsen}. 
Then $\V$-distributivity~\eqref{eq:eq-ccd2} always holds in $\V$, but~\eqref{eq:V-eq-ccd}, which implies lattice-distributivity, does not hold in the non-distributive lattices $M_3$ and $N_5$. 
%
Let us consider for example $\V = M_3$. 
Put $K=\{0,1\}$ and take the constant trivial weight $\psi(0)=\psi(1)=e$. Also, consider the following $K$-indexed family of $\V$-subsets of $\V$ itself, seen as an $\V$-category:
\[
G(k)(x) = 
\begin{cases}
e & k=0, x=e \mbox{ or } k=1, x\in \{a,b\}
\\
\bot & \mbox{otherwise}
\end{cases} \ . 
\]

Equation~\eqref{eq:V-eq-ccd} then becomes $e\wedge(a\vee b) = (e\wedge a) \vee (e\wedge b)$, which interpreted in $\V=M_3$ is false.

During the evaluation the right-hand side of \eqref{eq:V-eq-ccd}, observe that the ordered downsets $G(0)^{\downarrow}$ and $G(1)^{\downarrow}$ are, respectively, $\{\bot,e\}$  and $\{\bot, a,b\}$, hence their intersection is adequately computed by $ (e\wedge a) \vee (e\wedge b) = \bot$. 
But to compute the ``enriched meet'', we need to go back to the formula of the ``weighted downsets''~\eqref{eq:down-closure}. 
{\small{
\begin{align*}
&
\begin{array}{|c||c|c|c|c|c|}
\hline
[-,-] & \bot & a    & e    & b    & \top
\\ \hline\hline
\bot  & \top & \top & \top & \top & \top
\\
a     & b    & \top & b    & b    & \top 
\\
e     & \bot & a    & e    & b    & \top
\\
b     & a    & a    & a    & \top & \top
\\
\top  & \bot & a    & \bot & b    & \top
\\ \hline
\end{array}
&
&\begin{array}{|c||c|c|c|c|c|}
\hline 
x                    & \bot & a & e & b & \top
\\ 
\hline\hline
 G(0)^{\downarrow}(x) & \top & b & e & a & \bot
\\ 
\hline
G(1)^{\downarrow} (x) & \top & \top & \top & \top &  \top\\
\hline
\end{array}
\end{align*}
}}
\hskip-1ex The table above on the left exhibits the internal hom for the quantale structure of $M_{3}$ of Example~\ref{ex:quantales}\ref{exle:Jipsen}.
Therefore $G(0)^{\downarrow}(x) = [x,e]$ and $G(1)^{\downarrow}(x) = [x,a]\vee [x,b]$. 
We see that the enrichment does change $G(1)^{\downarrow}$ significantly. 
One reason is 
-- 
as for example the calculation $G(1)^{\downarrow}(e)=\left(G(1)(a)\tensor [e,a]\right) \vee \left(G(1)(b) \tensor [e,b] \right) =[e,a]\vee[e,b] = a\vee b = \top$ shows
-- 
that due to the tensor products $e\tensor a = a$ and $e\tensor b = b$ the element $e$ can be ``drawn into'' the downset $\{\bot, a,b\}$ even if $e$ is not below $a$ or $b$. 
Another reason, is that while $[\top,a]$ evaluates to zero over $\two$, it evaluates to $[\top,a]=a$ over $M_3$, so that $G(1)^{\downarrow}(\top)=[\top,a]\vee[\top,b]=a\vee b=\top$. 
In fact, the only difference to the right-hand side of \eqref{eq:V-eq-ccd} is the presence of the $\V$-down-closure in the above calculation, but these, read in the appropriate enriched way, do enlarge the downsets just enough to make the equality work again.
\end{example}


\section*{acknowledgements}

The authors wish to thank the organisers of the special session {\sf Mathematical Structures in Formal System Development and Analysis} of the {\em 9th Congress of Romanian Mathematicians (2019)} for inviting them to present their work at the congress.



\end{document}